\documentclass[12pt]{article}
\usepackage{amsmath,amsthm,amssymb}
\usepackage{multirow}
\usepackage{graphicx,epstopdf}

\newtheorem{theorem}{Theorem}[section]
\newtheorem{lemma}[theorem]{Lemma}

\numberwithin{equation}{section}
\numberwithin{table}{section}

\newcommand{\ssection}[1]{
     \section{\normalsize\bf #1}}

\begin{document}

\title {\large\sc Solving Partial Differential Equations on Closed Surfaces\\
with Planar Cartesian Grids
 }

\author{
{\normalsize J. Thomas Beale} \\
{\normalsize{\em Department of Mathematics, Duke University, Box 90320}}\\
{\normalsize{\em Durham, North Carolina 27708, U.S.A.}}\\
{\normalsize{\tt beale@math.duke.edu}}
}

\date{}

\maketitle

\newcommand{\beq}{\begin{equation}}
\newcommand{\eeq}{\end{equation}}
\newcommand\p{\,+\,}
\newcommand\lee{\,\leq\,}
\newcommand\gee{\,\geq\,}
\newcommand\m{\,-\,}
\newcommand\eq{\,=\,}
\newcommand{\eps}{\varepsilon}
\newcommand{\sig}{\sigma}
\newcommand{\pa}{\partial}
\newcommand{\lilhalf}{{\textstyle \frac12}}
\newcommand{\NN}{{\mathcal N}}
\newcommand{\R}{{\mathbb R}}
\newcommand{\Z}{{\mathbb Z}}

\begin{abstract}
We present a general purpose method for solving partial differential equations on a closed surface, based on a technique for discretizing the surface introduced
by Wenjun Ying and Wei--Cheng Wang [J.\ Comput.\ Phys. 252 (2013), pp. 606–-624]
which uses projections on coordinate planes.  Assuming it is given as a
level set, the surface
is represented by a set of points at which it intersects the intervals
between grid points in a three-dimensional grid.
They are designated as primary or secondary.  Discrete functions
on the surface have independent values at primary points, with values at secondary 
points determined by an equilibration process.  Each primary point and
its neighbors have projections to regular grid points in a coordinate plane where
the equilibration is done and finite differences are computed.
The solution of a p.d.e. can be reduced to standard methods on Cartesian grids in the coordinate planes, with the equilibration allowing seamless transition from one system to another.  We observe second order accuracy in examples with a variety of equations, including surface diffusion determined by the Laplace-Beltrami operator and the shallow water equations on a sphere.

\medskip



\end{abstract} 

\ssection {\bf Introduction} 
We present a general purpose method for solving partial differential equations on a closed surface, based on a technique for discretizing the surface introduced by Wenjun Ying and Wei-Cheng Wang \cite{YW13}, which uses projections on coordinate planes. Assuming it is given as a level set, the surface is 
represented by a set of points at which it intersects the intervals between grid points in a three-dimensional grid.  We will call these {\it cut points}.  They are designated as primary or secondary.  Discrete functions on the surface have independent values at primary points, and values at secondary points are determined from those by an equilibration process.  Each primary point and its neighbors have projections to regular grid points in a coordinate plane where computations can be done, including the equilibration and calculation of finite differences.  The solution of 
a p.d.e. can be reduced to standard methods on Cartesian grids in the coordinate planes, with the equilibration allowing seamless transition from one coordinate system to another.  We observe second order accuracy in examples with a variety of equations.

To determine the primary points, we find the closest grid point to each cut point.  Among those cut points which share the same closest grid point, we designate the one closest to the grid point as primary and any others as secondary.  (See Figure 1.)  This choice has the effect of determining which coordinate plane should be used.  If, for example, a primary point has the form $(ih,jh,z(ih,jh))$, then
the component of the normal vector in the $z$-direction is the largest in magnitude, within tolerance $O(h)$; see 
Lemma 2.1 below.  Thus the surface has the form $z = z(x,y)$ near this primary point, and calculations are done on regular grid points in the $xy$-plane using projections of neighboring cut points on vertical grid intervals.  However, some of these neighbors may be secondary points.  The primary points are distributed quasi-uniformly; see Lemma 2.2.  The procedure for selecting primary and secondary points and listing neighbors and coefficients for finite differences and equilibration is explained in the next section.


\begin{figure}[t!]
\centering
\includegraphics[height=2.1in]{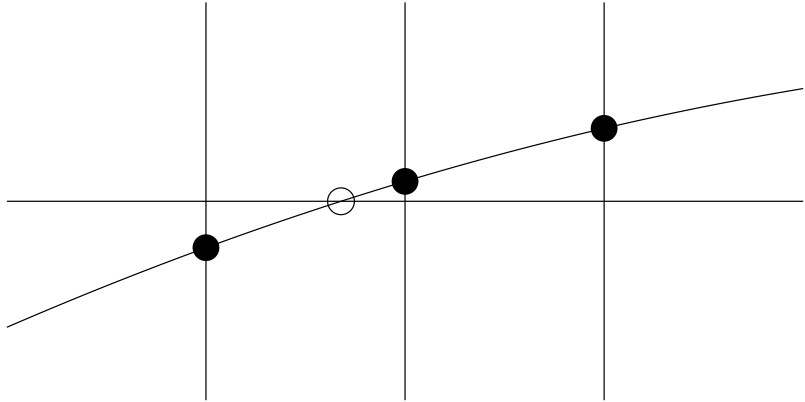}
\caption{Primary (filled) and secondary (unfilled) cut points for a planar curve.}
\end{figure}

To solve partial differential equations on a surface using this approach we need to write the equations in each of three coordinate systems: 
(1) $(y,z)$, with $x = x(y,z)$; (2) $(z,x)$, with $y = y(z,x)$; and
(3) $(x,y)$, with $z = z(x,y)$.
To perform an explicit update in a time-dependent p.d.e.,
we update the unknown at each primary point in the appropriate
coordinate system.  We then equilibrate to obtain the updated
values at the secondary points.  The examples presented here illustrate how
standard differential operators can be expressed in the coordinates, including
the Laplace-Beltrami operator, the surface gradient and surface divergence.
Only basic differential geometry is needed for this special choice of coordinates.  
We begin in Section 3 with a diffusion equation or a Poisson equation on a surface in which the differential operator is the Laplace-Beltrami operator.  The discrete LB operator has a nine-point stencil in the coordinate plane.  We discuss
the spectrum and resolvent.

In Section 4 we apply this method to a linear advection equation and to the shallow water equations on a sphere.  Formulas are obtained in the coordinate systems for
transport by a tangential velocity field, surface gradient and divergence, and
the material or substantial derivative,
writing the tangential velocity as a three-vector.
The formulas are not drastically different from those in the plane.
We solve the equations for a standard test problem with a version of the
Lax-Wendroff method.  Nonlinear hyperbolic equations such as the shallow water equations are difficult to solve accurately and require more careful methods;
here we have only the limited goal
of establishing the feasibility of the present approach for such problems.

It is often desirable to monitor integrals of evolving quantities, such as
conserved mass.  In the present setting surface integrals can be computed
in a natural way introduced in
\cite{Wilson} and \cite{byw}, using the cut points as quadrature points.
This rule is summarized in Appendix B.

Often numerical methods for p.d.e.'s on surfaces use triangulation and finite elements \cite{dziuk}.  Other methods use extension of the p.d.e. to a neighborhood of the surface \cite{varimpl,pnas05}.  The closest point method \cite{cpell,closepnas,closeimpl,closemove,close08} uses an extension such that e.g. the LB operator becomes the usual Laplacian;
it can apply to point clouds and to moving surfaces.  
In \cite{vv} a variational principle on the surface is extended to a neighborhood and an extended p.d.e. is derived.
The method of \cite{llz,zhao13,zhao18} is also suitable for point clouds; it uses a local coordinate grid near each point to find the surface and derivatives of the unknown using least squares.  Radial basis functions have also been used to discretize the LB operator \cite{rbffogel}.

The shallow water equations are used as a depth-averaged model of global atmospheric motion.  Because of their fundamental importance, a large amount of work has been devoted to their accurate solution using a variety of approaches.  Reviews with emphasis on gridding include \cite{staniforth,wmrev}.  An approach using
Cartesian grids on the sphere was given in \cite{cartsph}.  The velocity is
treated as a three-vector in some work including \cite{logrect,hesthaven,cartsph}. 
Riemann solvers were used e.g. in \cite{logrect,callevrev,ullrich}.  Radial basis functions are also used \cite{fornfly}. 

The present method is quite simple and direct.  Provided the surface is fairly smooth and known as a level set, the solution of a p.d.e. is performed with conventional finite differences on regular two-dimensional grids without boundary conditions.  Variables do not need to be extended beyond the surface.  In the applications given here, the projected problems are not very different from standard ones in the plane.  It appears that familiar numerical methods in the plane can be used on surfaces when combined with this discretization.  Possible further applications are discussed briefly in Section 5.  
However, this approach could not easily be adapted to a surface given as a point cloud, as in \cite{zhao13,closepnas,zhao18}, or a surface with varying spatial scales, as in \cite{zhao13,zhao18}.  The surface discretization used here was introduced in \cite{YW13} in order to solve boundary value problems in the integral equation formulation, replacing the computation of boundary integrals with a finite difference method for equivalent interface problems.

Although some analysis given here supports the validity of these methods, it seems difficult to prove convergence.  The equilibration appears to obstruct the use of standard arguments based on the maximum principle or summation by parts.  We hope this challenge can be met in the future.

We describe the surface discretization in detail in Section 2.  In Section 3 we formulate the discrete Laplace-Beltrami operator and apply it to a diffusion equation on various surfaces as well as the corresponding Poisson equation.  We compute the lowest eigenvalues of the discrete LB operator on the sphere and present some partial information about the resolvent.  In Section 4 we first solve a linear advection equation on a sphere with a known exact solution.  We then express several differential operators in the coordinates and formulate the shallow water equations on a sphere.  We compute the solution for a test problem from \cite{testset}.  We conclude with some discussion in Section 5 of possible further work.  In Appendix A it is proved that the discrete LB operator has positive resolvent for the special case of a curve in $\R^2$.  The
quadrature rule for surface integrals is given in Appendix B.  Calculations were done in Matlab.

\ssection { Surface discretization}  
We outline the procedure for discretizing the surface, give some details, and then explain the equilibration process.  (Cf. \cite{YW13}, Sections 3--5.)  We assume the surface has the form $\phi(x,y,z) = 0$ with the level set function $\phi$ known at least at the grid points in $\R^3$.  We assume $\phi$ is $C^2$ and $|\nabla\phi| \geq c_0$ on the surface for some $c_0 > 0$.
The steps are these:

1. Choose a three-dimensional grid with size $h$ covering the surface.  Label each grid point as inside or outside the surface according to the sign of $\phi$.

2. Find the grid intervals, e.g. from $(ih,jh,kh)$ to $(ih,jh,(k+1)h)$,
 with one grid point inside and one outside.

3. Find a cut point on the surface in each interval found in step 2,
or in a restricted subset of intervals.
Form sets of cut points 
$\Gamma_\nu$,  $\nu = 1,2,3$, where those in $\Gamma_3$ have the form
$(ih,jh,z(ih,jh))$ etc.

4. Assign to each cut point the closest grid point.

5. For each grid point, designate the closest cut point assigned to it in step 4, if any, as primary and any others as secondary.  Assign to each secondary point its associated primary point.  (See Fig. 1 and \cite{YW13}, Figs. 2 and 3.)

6. For each primary point in $\Gamma_\nu$ list the neighboring
cut points in $\Gamma_\nu$, some of which may be secondary.  E.g., if
$(ih,jh,z(ih,jh)) \in \Gamma_3$, include $(i'h,j'h,z(i'h,j'h))$
with $|i'-i|\leq 1$ and $|j'-j|\leq 1$.  

7. For each secondary point, find (at least) three points and coefficients needed for quadratic interpolation.  The center point is the associated primary point, and the other two are its neighbors in the same set $\Gamma_\nu$.  E.g., if $\bf s$ is a secondary point in $\Gamma_1$ with the form ${\bf s} = (x,jh,kh)$, with closest grid point $(ih,jh,kh)$, and $(ih,jh,z(ih,jh)) \in \Gamma_3$ is the associated primary point, the two neighbors are
$(ih\pm h,jh,z^\pm)$.  The coefficients are in (2.1) below.


In step 1, if a grid point happens to be on the surface, it can be assigned inside or outside.  In step 3, we do not need to find all secondary points, as explained below.
In step 3, if the cut point is a grid point, we need more information to determine which set $\Gamma_\nu$ it should belong to.  We could do this by determining the largest component of the normal vector
${\bf n} = \pm \nabla\phi/|\nabla\phi|$ or simply match with a neighboring primary point which is not a grid point. In steps 4 and 5, in case of equal distances we can choose arbitrarily.

It is not necessary, and could be difficult, to find all secondary points.  To avoid this, we can choose $\eta < 1/\sqrt{3}$ and find only those cut points in $\Gamma_\nu$ such that $|n_\nu| \geq \eta$, where $n_\nu$ is the $\nu$-component of the unit normal to the surface. 
The cut points omitted in this way are not needed as neighbors of the primary points.  We can reject these points by excluding intervals in step 2 for
which $|n_\nu| < \eta$ at an endpoint.
 The admissible cut points are well separated:  If $h$ is small enough,
depending on the first two derivatives of the surface, a grid interval can have at most one admissible cut point.  The admissible points can be found by a simple line search.  These facts were shown in \cite{Wilson}, pp. 11-18.  

A discrete function on the surface will have independent values at the primary points, with values at the secondary points determined by those.  We now explain the equilibration process that produces the values at the secondary points.
Suppose as in step 7 that 
${\bf s} \in \Gamma_1$ is a secondary point of the form ${\bf s} = (x,jh,kh)$, with
closest grid point $(ih,jh,kh)$, so that $x = ih + \theta h$ with $|\theta| \leq \lilhalf$.   If the associated primary point is ${\bf p} \in \Gamma_3$, then
${\bf p} = (ih,jh,z)$ for some $z$.  The function value at ${\bf s}$ should agree with that interpolated from values at cut points ${\bf q^\pm} = (ih\pm h,jh,z^\pm) \in \Gamma_3$ and $\bf p$.  Given values $u({\bf p}), u({\bf q^\pm})$ we require
$u({\bf s})$ to be the quadratic interpolation
\beq u({\bf s}) = \lilhalf(-\theta + \theta^2)u({\bf q^-}) 
    + (1 - \theta^2)u({\bf p}) + \lilhalf(\theta + \theta^2)u({\bf q^+}) \eeq
However,
since either or both of ${\bf q^\pm}$ may be secondary, we cannot simply interpolate secondary values from primary values.  Instead, as in \cite{YW13}, Sec. 5.1,  we solve a system of equations so that conditions such as (2.1) hold for all secondary points.  The system of
equations has the form 
\beq u^s = \Pi_{sp}u^p + \Pi_{ss}u^s  \eeq
where $u^s$ and $u^p$ are column vectors of values on secondary and primary points, resp. and $\Pi_{sp}$, $\Pi_{ss}$ are matrices.  The $i$th row of (2.2) gives the equation (2.1) for the $i$th secondary point.
Each row of $\Pi_{ss}$ has at most two nonzero entries with absolute sum at most
$\lilhalf(|\theta| + \theta^2) + \lilhalf(|\theta| - \theta^2) = |\theta| \leq \lilhalf$.  Thus $I - \Pi_{ss}$ is strictly diagonally dominant and invertible.  In particular equation (2.2)
can be solved by the iteration
\beq (u^s)^{n+1} = \Pi_{sp}u^p + \Pi_{ss}(u^s)^n  \eeq
but in this work we
invert the sparse matrix $I - \Pi_{ss}$, obtaining
\beq u^s = (I - \Pi_{ss})^{-1}\Pi_{sp}u^p  \eeq

\medskip

The following two lemmas make precise the facts that the normal to the surface at a primary point in $\Gamma_3$ is largely in the $z$-direction, with $z = z(x,y)$ nearby on the surface, and the primary points are well distributed.

\begin{lemma}
If ${\bf p} \in \Gamma_3$ is a primary point and the normal vector at $\bf p$ is $\bf n(p)$,
then $|n_\nu({\bf p})| \leq (1 + Ch)|n_3({\bf p})|$ for $\nu = 1,2$, where $C$ depends on the first two
derivatives of $\phi$.  On the surface near ${\bf p}$, $z$ is a function of $(x,y)$
with $|z_x|, |z_y| \leq (1 + Ch)$, and similarly for $\Gamma_1$, $\Gamma_2$.
\end{lemma}

\begin{proof}
For convenience assume ${\bf p} = (0,0,\theta h)$ with $0 < \theta \leq \lilhalf$.  With
$\nabla \phi = (\phi_x,\phi_y,\phi_z)$, suppose $|\phi_x| \geq |\phi_y|$ at $\bf p$.
Since $\bf{n(p)} = \pm \nabla \phi/|\nabla \phi|$, it will be enough to show that
$|\phi_x| \leq (1+ Ch)|\phi_z|$, with $C$ to be chosen.  If this is not true,
$|\phi_x| \geq |\nabla\phi|/\sqrt{3}$
and $|\phi_z|/|\phi_x| < (1 + Ch)^{-1}$.  Near $\bf p$, $x$ is a function of $z$ along the curve $\phi(x,0,z) = 0$, with $x(\theta h) = 0$ and
$x'(\theta h) = -\phi_z/\phi_x$ at $\bf p$.  Also $x''$ is bounded, so that
$|x'(z)| \leq (1 + Ch)^{-1} + C_2 h$ near $\bf p$.  We choose $C \geq 2C_2$ so that $|x'(z)| < 1$.
Then $|x(\theta h) - x(0)| < \theta h$, or $|x(0)| < \theta h$, and thus $(x(0),0,0)$ is a cut point closer to $(0,0,0)$ than
$\bf p$, contradicting the fact that  $\bf p$ is primary.  The second statement follows from the first and
the equalities $z_x = -\phi_x/\phi_z$, $z_y = - \phi_y/\phi_z$. 
\end{proof}

\begin{lemma}
  The set of primary points is quasi-uniform:  Any two primary points are separated by distance
at least $(\sqrt{2}/2)h - Ch^2$ with some constant $C$.  For a primary point ${\bf p} = (ih,jh,z(ih,jh)) \in \Gamma_3$,
there is a primary point within distance $(3\sqrt{2}/2)h + Ch^2$ of $\bf p$ which belongs to the same grid point as
the cut point ${\bf p^+} = ((i+1)h,jh,z((i+1)h,jh)) \in \Gamma_3$, and similarly for other neighbors of $\bf p$ and for
$\Gamma_1$, $\Gamma_2$.
\end{lemma}

\begin{proof}
Suppose ${\bf p} = (0,0,\theta h)$ is a primary point belonging to the grid point $(0,0,0)$, with $0 \leq \theta \leq \lilhalf$.  Any other point in
$\Gamma_3$ is at least distance $h$ away.  Any primary point $(x^*,jh,kh) \in \Gamma_1$ is at distance at least $h$
provided $j \neq 0$ or $k \neq 0$ or $1$, and similarly for $\Gamma_2$.  
We are left to consider possible primary points of the form $(x^*,0,0)$, $(x^*,0,h)$, $(0,y^*,0)$, or $(0,y^*,h)$.
In the first case, if $|x^*| < h/2$, a cut point $(x^*,0,0)$
belongs to the grid point $(0,0,0)$ and must be secondary.  Suppose $(x^*,0,0)$ is a primary point
with $|x^*| \geq h/2$.  Then $|\phi_z| \leq (1 + Ch)|\phi_x|$ as in Lemma 2.1.  Also, because $\bf p$ is primary,
$z$ is a function of $x$ on the curve $\phi(x,0,z) = 0$, with $z' = -\phi_x/\phi_z$, so that $|z'| \geq 1 - C_1h$.
With $z(0) = \theta h$ and $z(x^*) = 0$, we have $|\theta h - 0| \geq (1 - C_1 h)|x^*|$.  Then $|(0,0,\theta h) - (x^*,0,0)|
\geq |x^*| (1 + (1 - C_1 h)^2)^{1/2} \geq (h/2)(\sqrt{2} - C_2h)$, with some $C_2$, and this is equivalent to the conclusion stated.
For a cut point $(x^*,0,h)$ we can argue similarly, using the bound $|z_x| \leq (1 + Ch)$ and the fact that $h - \theta h \geq h/2$.
Other cases are analogous to these.

For the second statement, we have $z = z(x,y)$ on the surface with $|z_x| \leq (1 + Ch)$ as before,
so that ${|\bf p^+ - p|}\leq (\sqrt{2} + C'h)\,h$.  If ${\bf p^+}$ is secondary, the primary point belonging to the same grid
point is within distance $(\sqrt{2}/2)h$ of $\bf p^+$, and the conclusion follows.
\end{proof}

\ssection {\bf The Laplace-Beltrami operator and diffusion equations}


\noindent{\bf The discrete Laplace-Beltrami operator.}
The Laplace-Beltrami operator, or surface Laplacian, is the generalization of the usual Laplacian to a surface or manifold.  For the unit sphere it is the angular part of the Laplacian.  It is invariant under change of coordinates and is expressed in any coordinate system in terms of the metric tensor.  Suppose
$(\xi_1,\xi_2)$ are coordinates on a surface $\Gamma \subseteq \R^3$ so
that part of $\Gamma$ is given as ${\bf X}(\xi_1,\xi_2)$.  At each point we have tangent vectors ${\bf X}_i = \pa{\bf X}/\pa\xi_i$, $i=1,2$, metric tensor $g_{ij} = {\bf X_i}\cdot{\bf X_j}$, the inverse $g^{ij} = (g_{ij})^{-1}$, and $g = \det{g_{ij}}$.  The Laplace-Beltrami operator applied to a scalar function
$u$ on $\Gamma$ is
\beq \Delta u \eq \frac{1}{\sqrt{g}}\sum_{i,j = 1}^2 \frac{\pa}{\pa\xi_i}
  \left( \sqrt{g}g^{ij} \frac{\pa u}{\pa\xi_j} \right)  \eeq

We can discretize the LB operator in divergence form on a regular grid using a stencil with 9 points.  We will describe the discrete operator $\Delta_h$, with grid size $h$,  at a point labeled $(\xi_1,\xi_2) = (0,0)$ for convenience, using
points $(\sigma_1 h, \sigma_2 h)$ with $\sigma_1, \sigma_2 = -1,0,1$.
We will set $a^{ij} = \sqrt{g}g^{ij}$ and use notation such as
$$ a^{ij}_{+-} = (a^{ij}(h,-h) + a^{ij}(0,0))/2 $$
Because of the mixed derivatives, we
begin with the second difference along a diagonal,
\beq L_\diagup (a^{12}, u) \eq 
  \left( a^{12}_{++}(u(h,h) - u(0,0)) - a^{12}_{--}(u(0,0) - u(-h,-h))\right)/h^2 \eeq
To approximate (3.2) we use the Taylor expansion
\beq a^{12}_{++} \eq a^{12} + h(a^{12}_1 + a^{12}_2)/2 + 
       h^2(a^{12}_{11} + a^{12}_{22} + 2a^{12}_{12})/4 + O(h^3) \eeq
and the corresponding formula for $a^{12}_{--}$,
where the quantities on the right are evaluated at $(0,0)$ and subscripts denote derivatives.  We find
\beq \left(a^{12}_{++} + a^{12}_{--}\right)/2 = a^{12} + O(h^2) \,,\qquad
    \left(a^{12}_{++} - a^{12}_{--}\right) = h(a^{12}_1 + a^{12}_2) + O(h^3) \eeq
Similarly we have
\beq u(\pm h, \pm h) \eq u \;\pm\; h(u_1 + u_2)
+ h^2(u_{11} + u_{22} + 2u_{12})/2     
                \pm c_3\, h^3 + O(h^4)  \eeq
from which we get
\begin{multline}
  u(h,h) + u(-h,-h) - 2u(0,0) = h^2(u_{11} + u_{22} + 2u_{12}) + O(h^4) \,,
       \\ u(h,h) - u(-h,-h) = 2h(u_1 + u_2) + O(h^3)
\end{multline}
We rearrange (3.2) and use (3.4) and (3.6) to obtain
\begin{multline}
    L_\diagup (a^{12}, u) = 
\left(a^{12}_{++} + a^{12}_{--}\right)(u(h,h) + u(-h,-h) - 2u(0,0))/2h^2 \\
   + \left(a^{12}_{++} - a^{12}_{--}\right)(u(h,h) - u(-h,-h))/2h^2 \\
       \eq a^{12}(u_{11} + u_{22} + 2u_{12})
          \p (a^{12}_1 + a^{12}_2)(u_1 + u_2) + O(h^2)
\end{multline}
We also use the standard second differences in direction $i = 1$ or $2$
\begin{gather}
   L_1({\tilde a}^{11}, u) \eq 
\left({\tilde a}^{11}_{+0}(u(h,0) - u(0,0)) 
   - {\tilde a}^{11}_{-0}(u(0,0) - u(-h,0))\right)/h^2 \\
   L_2({\tilde a}^{22}, u) \eq 
\left({\tilde a}^{22}_{0+}(u(0,h) - u(0,0)) 
   - {\tilde a}^{22}_{0-}(u(0,0) - u(0,-h))\right)/h^2 
\end{gather}
with ${\tilde a}^{ii} = a^{ii} - a^{12}$.  As above we find
\beq  L_i(a^{ii}- a^{12}, u) \eq (a^{ii} - a^{12})u_{ii} 
      \p (a^{ii}_i - a^{12}_i) u_i
       \p O(h^2) \eeq
We now define
\beq \Delta_h u \,\equiv\, \left( L_1(a^{11} - a^{12}, u) + 
     L_2(a^{22} - a^{12}, u)  \p L_\diagup (a^{12}, u)\right)/\sqrt{g}  \eeq
From (3.7),(3.10) we have, after combining terms,
\begin{multline}
 \sqrt{g}\,\Delta_h u \eq a^{11}u_{11} + a^{22}u_{22} + 2a^{12}u_{12} \\
   + a^{11}_1u_1 + a^{22}_2u_2 + a^{12}_1u_2 + a^{12}_2u_1 + O(h^2)\\
  \eq  (a^{11}u_1)_1 + (a^{22}u_2)_2 + (a^{12}u_1)_2  + (a^{12}u_2)_1 + O(h^2)
\end{multline}
Comparing with (3.1), we see that $\Delta_h u = \Delta u + O(h^2)$, as for
usual centered second difference operators.

We could alternatively use the other diagonal and define
\beq L_{\diagdown} (a^{12}, u) \eq 
  \left( a^{12}_{-+}(u(-h,h) - u(0,0)) - a^{12}_{+-}(u(0,0) - u(h,-h))\right)/h^2 \eeq
We find as above that
\beq L_{\diagdown} (a^{12}, u) \eq a^{12}(u_{11} + u_{22} - 2u_{12})
          \p (-a^{12}_1 + a^{12}_2)(-u_1 + u_2) + O(h^2) \eeq
Proceeding as before we define
\beq \Delta_h u \,\equiv\, \left(L_1(a^{11} + a^{12}, u) + 
     L_2(a^{22} + a^{12}, u) \m L_\diagdown (a^{12}, u)\right)/\sqrt{g} \eeq
It is again accurate to $O(h^2)$.
We use (3.11) if  $g^{12} \geq 0$ and (3.15) if $g^{12} \leq 0$.

In this work the coordinates are always $(y,z)$, $(z,x)$, or $(x,y)$, where
${\bf X} = (x,y,z)$.  In the third case, for example, $z = z(x,y)$ and
\beq g = 1 + z_x^2 + z_y^2 \,, \quad g^{11} = (1+z_y^2)/g \,, \quad
   g^{12} = - z_xz_y/g \,, \quad  g^{22} = (1 + z_x^2)/g  \eeq
For points in $\Gamma_3$ we have $|z_x|, |z_y| \leq 1 + O(h)$, according to Lemma 2.1.  It follows that  $g^{ii} - |g^{12}| > 0$, and the off-diagonal coefficients in (3.11),(3.15) are $\geq 0$ while the coefficient of $u(0,0)$ is $< 0$, assuming $g^{12} \geq 0$ or $g^{12} \leq 0$ respectively.

For a familiar surface, such as the sphere of radius $r$, we may prefer to replace  (3.11),(3.15) with a non-divergence form,
\begin{gather}
   \Delta_h u \eq (g^{11} - g^{12})D^2_1 u \p (g^{22} - g^{12})D^2_2 u
             \p g^{12}D^2_\diagup u  \p b_1 D_1 u + b_2 D_2 u \\
   \Delta_h u \eq (g^{11} + g^{12})D^2_1 u \p (g^{22} + g^{12})D^2_2 u
             \m g^{12}D^2_\diagdown u  \p b_1 D_1 u + b_2 D_2 u 
\end{gather}
where $D^2_\diagup u = L_\diagup(1,u)$ etc.,
with centered differences in the last two terms.
For the sphere with $z = z(x,y)$,
\begin{gather}
 g = r^2/z^2\,,\quad g^{11} = (r^2-x^2)/r^2\,,
         \quad g^{22} = (r^2-y^2)/r^2\,,\quad g^{12} = -xy/r^2\,, \\
  b_1 = - 2x/r^2\,, \quad b_2 = - 2y/r^2
\end{gather}

\medskip


\noindent{\bf The diffusion equation.}
The prototype equation for diffusion on a surface $\Gamma$ is
\beq u_t \eq \alpha \Delta u \quad \mbox{on\;} \Gamma  \eeq
where $u(\cdot,t)$ is function on $\Gamma$, $\Delta$ is the Laplace-Beltrami operator and $\alpha$ is a coefficient.
To solve this equation with a specified initial state we discretize the surface as in
Section 2 and replace $u$ with a function $u_h$ on $\Gamma_h$, the set of
cut points, both primary and secondary.  We also use the restriction $u^p_h$ to the set of primary points $\Gamma^p_h$.  We think of both as column vectors.  We calculate
the discrete LB operator $\Delta_h$ applied to
$u_h$ at points in $\Gamma^p_h$.  Thus $\Delta_h$ is an $n_p\times
n_{tot}$ matrix where $n_p$ is the number of primary points and $n_{tot}$ is
the number of all admissible points, primary and secondary.  The values of
$\Delta_hu_h$ at primary points are found using either (3.11),(3.15) or (3.17),(3.18) in the three coordinate systems.
We can extend
any function $f^p_h$ on $\Gamma^p_h$ to the remaining secondary points in $\Gamma_h$
by the equilibration $E_h$ defined in (2.4), i.e.,
\beq  E_h f^p_h\eq (f^p_h\;f^s_h)^T\,,\qquad 
       f^s_h \eq (I -\Pi_{ss})^{-1}\Pi_{sp}f^p_h  \eeq
so that $E_h$ is an $n_{tot}\times n_p$ matrix.

To solve (3.21) we replace $\Delta u$ by $E_h\Delta_h u_h$ and select a time stepping method.  We have used the forward Euler method and BDF2, the second order backward difference formula, as representatives of explicit and implicit methods.  With
time step $k$, the forward Euler method approximates (3.21) with
\beq u^{n+1}_h \eq u^n_h + k\alpha E_h\Delta_h u^n_h \eeq
whereas for BDF2 we have
\beq u ^{n+1}_h \eq
  (I - \frac23 k\alpha E_h\Delta_h)^{-1}(\frac43 u^n_h - \frac13 u^{n-1}_h) \eeq
We have not proved that the inverse matrix in (3.24) exists, but we obtain it in our
computations.  Properties of this resolvent matrix are discussed further below.
We could instead formulate the solution in terms of $u^p_h$ on $\Gamma^p_h$.
The forward Euler version in this case would be
\beq u^{p,n+1}_h \eq u^{p,n}_h + k\alpha \Delta_hE_h u^{p,n}_h \eeq
Assuming the initial state is equilibrated, then (3.23) and (3.25) are equivalent, i.e., if $u^{p,n}_h$ solves (3.25) then $u^n_h = E_h u^{p,n}_h$ solves (3.23).  Similarly, we find
that if $f_h$ on $\Gamma_h$ is equilibrated, with $f_h = E_hf^p_h$, and if
$v^p_h$ on $\Gamma^p_h$ satisfies
$ (I - \kappa\Delta_hE_h)v^p_h = f^p_h$, then $v_h = E_hv^p_h$ on $\Gamma_h$
satisfies $(I - \kappa E_h\Delta_h) v_h = f_h$.  Thus the resolvents of the
$n_{tot}\times n_{tot}$ matrix $E_h\Delta_h$ and the $n_p\times n_p$ matrix
$\Delta_h E_h$ are closely related.  We study the latter operator further below.

As a first example for the diffusion equation we choose $\Gamma$ to be the unit sphere
and the initial state to be a spherical harmonic, as in several references.  We take
\beq u^0(x,y,z) = 7(x - 2y)(15z^2 - 3)/8 \eeq
Then $\Delta u^0 = - 12u^0$ on the unit sphere.  We choose $\alpha = 1/12$ in (3.21) so that the exact solution is $u = e^{-t}u^0$.
We embed the sphere in a computational box $[-1.2,1.2]^3$ and introduce a grid with
$N$ intervals in each direction, so that $h = 2.4/N$.
We discretize as in Section 2.
The parameter $\eta$ is $.45$.  We then solve to time $t = 1$ with either the forward Euler method (FE) or BDF2, started with a backward Euler step.  We use
either the divergence form of $\Delta_h$ in (3.11),(3.15) or the nondivergence form (3.17),(3.18).  For each $N$ we choose time step
$k = 8/N^2$ with FE and $k = 1/(2N)$ with BDF2.  Relative errors in the four cases are shown in Table 3.1.
The relative $L^2$ error is
\beq \|u^{comp} - u^{exact}\|/\|u^{exact}\|\,, \qquad
    \|w\|^2 = (1/n_{tot})\sum_{{\bf X} \in \Gamma_h} w({\bf X})^2 \eeq
The relative maximum error is a similar ratio of absolute maxima.
We see that the convergence is about second order in each case.
The errors are somewhat larger for the divergence form, but it can be
used for a general surface.

\begin{table}[ht]
{\footnotesize
\caption{Relative errors for diffusion on the unit sphere} 
\begin{center} 
\begin{tabular}{| c | c | c | c | c | c | c | c | c |} \hline
\multicolumn{1}{|c|}{} &
\multicolumn{2}{c|}{FE} & \multicolumn{2}{c|}{BDF2} &
    \multicolumn{2}{c|}{FE} & \multicolumn{2}{c|}{BDF2}\\
\multicolumn{1}{|c|}{} &
\multicolumn{2}{c|}{nondiv} & \multicolumn{2}{c|}{nondiv} &
    \multicolumn{2}{c|}{div} & \multicolumn{2}{c|}{div}\\ \hline
N & max & $L^2$ & max & $L^2$ & max & $L^2$ & max & $L^2$ \\ \hline 
80 & 4.94e-4 & 3.21e-4 & 8.24e-4 & 4.64e-4 &
 1.01e-3 & 8.65e-4 & 1.45e-3 & 1.45e-3 \\ \hline
160 & 1.03e-4 & 5.92e-5 & 1.90e-4 & 1.32e-4 &
 2.44e-4 & 2.27e-4 & 3.67e-4  & 3.77e-4 \\ \hline
320 & 1.96e-5 & 1.67e-5 & 2.21e-5 & 2.44e-5 &
 4.96e-5 & 5.46e-5 & 8.65e-5 & 9.04e-5\\ \hline
640 & 5.03e-6 & 4.41e-6 & 4.53e-6 & 4.53e-6 &
 1.19e-5 & 1.35e-5 & 2.04e-5 & 2.05e-5 \\ \hline
\end{tabular} \end{center}   }
\end{table}

For our next examples we use two different surfaces.  The first is the ellipsoid
\beq x^2/a^2 \p y^2/b^2 \p z^2/c^2 \eq 1 \eeq
with $a = 1$, $b = .8$ and $c = .65$.  The second is obtained by rotating a Cassini oval about the $z$-axis,
\beq (x^2 + y^2 + z^2 + a^2)^2 \m 4a^2(x^2 + y^2) \eq b^4 \eeq
with $a = .65$ and $b/a = 1.1$, a nonconvex surface.  In both cases we solve (3.21) with $\alpha = .1$ and initial state $u^0(x,y,z) = \cos(x - y + z)$.  We compute
$\Delta_h$ as in (3.11), (3.15).  We use FE with
$k = 8/N^2$ and BDF2 with $k = 1/(10N)$ and solve to time $t = 1$.  Since we do not know the exact solution,
we measure errors by comparing successive runs.  With $N = 80$, $160$, $320$,
we compute successive $L^2$ errors
$ \eps_N = \|u^N - u^{2N}\| $,
with norm as in (3.27); similarly we find successive maximum errors.
Both are displayed in Table 3.2.
The convergence for the ellipsoid is about second order.  For the Cassini oval
the rate is less clear, but the decrease from $N = 80$ to $N = 320$ is more
rapid than second order.

\begin{table}[ht]
{\footnotesize
\caption{Successive errors for diffusion on two surfaces} 
\begin{center} 
\begin{tabular}{| c | c | c | c | c | c | c | c | c |} \hline
\multicolumn{1}{|c|}{} & \multicolumn{4}{c|}{Ellipsoid} & 
\multicolumn{4}{c|}{Cassini Oval} \\ 
\multicolumn{1}{|c|}{} &
\multicolumn{2}{c}{FE} & \multicolumn{2}{c|}{BDF2} &
    \multicolumn{2}{c}{FE} & \multicolumn{2}{c|}{BDF2}\\ \hline
N & max & $L^2$ & max & $L^2$ & max & $L^2$ & max & $L^2$ \\ \hline 

80 & 2.35e-4 & 6.32e-5 & 2.53e-4 & 9.46e-5 &
   6.68e-4 & 1.68e-4 & 6.52e-4 & 1.88e-4 \\ \hline
160 & 6.47e-5 & 1.73e-5 & 7.15e-5 & 2.47e-5 & 
   9.09e-5 & 3.10e-5 & 9.07e-5 & 3.39e-5\\ \hline
320 & 1.56e-5 & 4.60e-6 & 2.06e-5 & 6.81e-6 & 
   2.72e-5 & 9.12e-6 & 2.33-5 & 9.76-6\\ \hline
\end{tabular}  \end{center}   }
\end{table}

\medskip


\noindent{\bf Spectrum and resolvent of the discrete LB operator.}
We will call the $n_p\times n_p$ matrix
$\Delta^{red}_h = \Delta_h E_h$ the {\it reduced LB operator}.
It approximates the LB operator discretized to the primary points.  To compare it
to the exact operator, we compute its lowest eigenvalues on the unit sphere.
The exact spherical Laplacian has eigenvalues $\lambda_n = - n(n+1)$, $n \geq 0$,
with multiplicity $2n+1$.  Table 3.3 gives the errors in the first 49 eigenvalues for $\Delta^{red}_h$ for various $N$.
For $1\leq n \leq 7$, the maximum absolute error is displayed for the $2n+1$ eigenvalues close to $\lambda_n$.  They appear to converge to second order, with larger errors for higher $n$.  Here
$\Delta^{red}_h$ was computed in divergence form.  The nondivergence form gives somewhat smaller errors, except for $n = 1$.  Eigenvalues were also computed in
\cite{zhao13}.

\begin{table}[ht]
{\footnotesize
\caption{Absolute errors in eigenvalues for the spherical Laplacian} 
\begin{center}
\begin{tabular}{ | c | c | c | c | c | c |} \hline
 $-\lambda$ & mult & 40 & 80 & 160 & 320 \\ \hline
0 & 1 & 2.2e-15 & 3.1e-14 & 1.2e-14  & 1.8e-14\\ \hline
2 & 3 &3.01e-3 & 7.64e-4 & 1.10e-4 & 3.77e-5 \\ \hline
6 & 5 & 3.19e-2 & 7.97e-3 & 1.99e-3 & 4.95e-4\\ \hline
12 & 7 & 6.93e-2 & 1.70e-2 & 3.82e-3 & 1.02e-3\\ \hline
20 & 9 & 1.87e-1 & 4.67e-2 & 1.17e-2 & 2.91e-3\\ \hline
30 & 11 & 3.37e-1 & 8.45e-2 & 2.07e-2 & 5.25e-3\\ \hline
42 & 13 & 7.16e-1 & 1.79e-1 & 4.49e-2 & 1.13e-2\\ \hline
\end{tabular}  \end{center}  }
\end{table}

As noted in \cite{zhao13} it is desirable for a discretization of $I - k\Delta$ to be an M-matrix.
A square matrix $A$ is called an M-matrix if the diagonal entries are $> 0$, the off-diagonal entries are $\leq 0$, and $A$ is strictly diagonally dominant.  It follows that $A^{-1}$ is positive, i.e., each entry of $A^{-1}$ is $\geq 0$.  (E.g. see \cite{hackbook}.)  These properties have important consequences for the stability of schemes for elliptic and parabolic equations.  For example, if $\Delta_h^0$
 is the usual five-point Laplacian on a rectangle in $\R^2$ with periodic boundary conditions, $I - k\Delta_h^0$ is an M-matrix for $k>0$.  The row sums of
$(I - k\Delta_h^0)^{-1}$ are $1$, since those
of $\Delta_h^0$ are $0$.  This fact and the positivity imply that the resolvent operator $(I - k\Delta_h^0)^{-1}$, acting on vectors in maximum norm, has norm $1$.  Consequently if the heat equation $u_t = \Delta_h^0 u$ is solved with backward Euler time steps, the maximum of the solution is nonincreasing in time. 

Our discrete LB operator $\Delta_h$, constructed as (3.11),(3.15), has the correct sign conditions, but the interpolation (2.1) has both signs.  For this reason,
$I - k\Delta_h^{red}$ fails to be an M-matrix in general.  Nonetheless, we see in our computations that $(I -k\Delta_h^{red})^{-1}$ often has nonnegative entries. 
Constants are null vectors for $\Delta^{red}_h$, since a constant on the primary points equilibrates to the constant, and $\Delta_h$
applied to a constant vector gives zero.  Thus the row sums of $\Delta^{red}_h$
are zero, and 
$(I -k\Delta_h^{red})^{-1}$ has row sums one.  If $(I -k\Delta_h^{red})^{-1} \geq 0$ then it has norm $1$ as in the familiar case above.
We observe that $(I -k\Delta_h^{red})^{-1} \geq 0$ for $k/h^2$ larger than $1$
but not for $k/h^2$ small.  We are not able to prove this
positivity for surfaces in $\R^3$, but we  give a proof
for closed curves in $\R^2$ in Appendix A,
since it gives some insight into the interaction between the discrete LB operator and the equilibration.  There are only a few results concerning properties of matrices
perturbed from $M$-matrices, e.g. \cite{bouchon,bramble}.

\medskip


\noindent{\bf The Poisson equation.}
For the exact LB operator, the Poisson equation    
\beq \Delta u \eq f \;\mbox{on\;} \Gamma  \eeq
has a solution if and only if $\int f\,dS = 0$, and the solution is
unique up to an arbitrary constant (e.g. see \cite{dziuk}).
We noted that for the reduced discrete operator $\Delta^{red}_h$ constants are null vectors.
The eigenvalue computation for
$\Delta_h^{red}$ suggests they are the only null vectors, so that
the range has codimension one and solutions of the discrete equation are again
unique except for a constant.  We can attempt to solve the equation (3.30),
given $f$ in the range, by augmenting the  discrete problem to
\beq \Delta_h^{red} u^h \p \beta{\bf 1}  \eq f^h \,, \qquad
    \sum u^h = 0 \eeq
where $f^h$ is the restriction to primary points, $\beta$ is an extra (scalar)
unknown, ${\bf 1}$ is a vector of $1$'s, and the sum is over $\Gamma_h^p$.
The extra term adjusts $f^h$ so that $f^h - \beta {\bf 1}$ can be 
in the range of $\Delta_h^{red}$.  We expect that $f^h$ is
close to being in the range of $\Delta_h^{red}$ and consequently
$\beta$ will be small.
Then $u^h$ should approximate the solution of (3.30) with $\Sigma u = 0$.  (Such a method was outlined
in {\cite{hackbook}, Sec. 4.7 for the Neumann problem.)

We test this approach to solving (3.30) with a known exact solution.  We define
$u = \cos(x/r + y/r -2z/r)$ on $\R^3$, with $r^2 = x^2 + y^2 + z^2$ and find the usual Laplacian $f$.
Then the LB operator applied to $u$ on the unit sphere will equal $f$.
We solve (3.31) on the unit sphere with this choice of $f$.
We compare the computed solution with $u - u_0$ where $u$ is the exact
solution and $u_0 = n_p^{-1}\Sigma u$.
(Note that $u_0$ depends on $h$.)  With $N = 80$, $160$, $320$, $640$
we find absolute maximum errors
$9.20e\!-\!4$, $2.35e\!-\!4$, $5.70e\!-\!5$, $1.40e\!-\!5$, showing $O(h^2)$ convergence.

\ssection {\bf Advection and flow on the sphere}


\noindent {\bf A linear advection equation.}
As a test problem for transport on a surface, we have constructed
a linear advection equation on the unit sphere with an exact solution.
We will use $\varphi$ and $\theta$ for longitude and latitude, resp., with
$0 \leq \varphi \leq 2\pi\,$, $-\pi/2 \leq \theta \leq \pi/2$, so that
\beq x = \cos{\varphi}\cos{\theta}\,, \quad y = \sin{\varphi}\cos{\theta}\,, \quad
          z = \sin{\theta} \eeq
We set 
\beq  r^2 \eq \cos^2{\theta} \eq x^2 + y^2 \eq 1 - z^2  \eeq
Our advection equation for unknown $\Phi(\varphi,\theta,t)$ is
\beq \frac{\pa\Phi}{\pa t} \p \frac{\pa\Phi}{\pa \varphi}
          \m \cos{\varphi}\cos^2{\theta}  \frac{\pa\Phi}{\pa\theta} \eq 0 \eeq
The vector field represents rotation about the vertical axis with oscillation in 
latitude.  The equation
can be solved by the method of characteristics.  As initial state we take
\beq \Phi(\varphi,\theta,0) \eq \cos^2{\theta} \eq r^2  \eeq  
The exact solution in rectangular coordinates is
\beq \Phi = \frac{r^2}{\left( z + y(1-\cos{t}) + x\sin{t} \right)^2 + r^2} \eeq
To solve the initial value problem with the present method we first rewrite the
p.d.e. (4.3) in each of the three coordinate systems, e.g. with $\Phi = \Phi(x,y)$
in $\Gamma_3$.  The three forms are
\begin{gather}
\Phi_t + (x + xyz)\Phi_y - xr^2\Phi_z \eq 0\,,\qquad x = x(y,z) \\
\Phi_t - xr^2\Phi_z + (x^2z - y)\Phi_x \eq 0\,,\qquad y = y(z,x) \\
\Phi_t + (x^2z - y)\Phi_x + (x + xyz)\Phi_y \eq 0\,, \qquad z = z(x,y)
\end{gather}
We discretize the three equations using the MacCormack two-step version of the 
Lax-Wendroff method.  With time step $k$, given $\Phi^n \approx \Phi(\cdot,nk)$
at all points, primary and secondary, the first step is to compute the update to $\Phi^n$ at primary points, equilibrate the update to secondary points, and add the update to $\Phi^n$ to obtain the predictor $\Phi^*$.  Thus e.g. in the third system we compute the update at primary points
\beq F^* \eq  \m
  (x^2z - y)D_x^+\Phi^n \m (x + xyz)D_y^+\Phi^n  \eeq
where $D_x^+$, $D_y^+$ are usual forward differences.  After doing the same in all three systems, we equilibrate as in (2.4) to extend $F^*$ from primary points to $E_hF^*$ defined on all points, with $E_h$ as in (3.22).  We then set
\beq \Phi^* \eq \Phi^n \p k E_h F^*  \eeq
The second step uses a similar procedure with backward differences, with the update in the third system at the primary points
\beq F^{n+1} = \m (x^2z - y)D_x^-\Phi^* 
      \m (x + xyz)D_y^-\Phi^*  \eeq
and after finding $F^{n+1}$ at primary points in all three systems
\beq \Phi^{n+1} \eq \frac12(\Phi^n + \Phi^*) \p \frac{k}{2}E_hF^{n+1} \eeq

We use grid size $h = 2.4/N$ and time step $k = 1/2N$
for various choices of $N$.  Relative errors are displayed in Table 4.1.
They are clearly $O(h^2)$.
The $L^2$ error is 
\beq
\frac{\left( \sum \left| \Phi^{comp}({\bf p}) - \Phi^{exact}({\bf p})\right|^2 \right)^{1/2} }
  {\left( \sum  |\Phi^{exact}({\bf p})|^2 \right)^{1/2} }
\eeq
where the sum is over all primary points and admissible secondary points.  Similarly
the relative maximum error is $\max |\Phi^{comp} - \Phi^{exact}| /
\max |\Phi^{exact}|$.  We also compute the surface integral of $\Phi$ using the method explained in Appendix B and display the relative error.  (The integral is not constant in time since ${\bf v}$ does not have divergence zero.)

\begin{table}[ht]
{\footnotesize
\caption{Relative errors for the linear advection equation} 
\begin{center} 
\begin{tabular}{| c | c | c | c | c |} \hline
 N & time & $\;$ max error $\;$ & $\;$ $L^2$ error $\;$ & error in $\int\Phi$
\\ \hline  \hline
 \multirow{3}{*}{80}
 & 1 & 3.29e-3  & 8.13e-4 &  1.47e-4
\\ \cline{2-5}
 & 2 &  7.12e-3 &  2.50e-3 &  1.70e-4
\\ \cline{2-5}
 & 5 &  3.32e-2 &  1.59e-2 & -2.26e-3
\\ \hline \hline
\multirow{3}{*}{160}
 & 1 & 7.59e-4  & 2.01e-4 &  3.63e-5
\\ \cline{2-5}
 &  2 &  1.76e-3 &  6.24e-4 &  4.19e-5
\\ \cline{2-5}
 & 5  &  8.32e-3  & 4.03e-3 & -5.69e-4
\\ \hline \hline
\multirow{3}{*}{320}
 & 1 &   1.79e-4 &   5.01e-5 &   9.12e-6
\\ \cline{2-5}
 & 2 &   4.37e-4  &  1.57e-4  &  1.04e-5
\\ \cline{2-5}
 & 5 &   2.13e-3 &   1.01e-3 &  -1.43e-4 
\\ \hline \hline
\multirow{3}{*}{640}
  &  1 &  4.43e-5  & 1.26e-5 &  2.28e-6
\\ \cline{2-5}
  & 2  & 1.10e-4  & 3.92e-5 &  2.64e-6
\\ \cline{2-5}
  & 5 &  5.34e-4  & 2.53e-4 & -3.59e-5
\\ \hline 
\end{tabular} \end{center}  }
\end{table}

\medskip

   
\noindent {\bf Differential operators on a surface.}
Before proceeding it will be helpful to interpret this example in a setting independent of coordinates and discuss the surface gradient and divergence.  The p.d.e. (4.3) in spherical coordinates is equivalent to 
\beq \frac{\pa\Phi}{\pa t} \p {\bf v}\cdot\nabla\Phi \eq 0  \eeq
where $\nabla\Phi$ is the surface gradient and
$\bf v$ is the tangential vector field 
\beq {\bf v} \eq {\bf X}_\varphi - \cos{\varphi}\cos^2{\theta}{\bf X}_\theta
     \eq (x^2z - y, x + xyz, -xr^2)^T \eeq
Here ${\bf X} = (x,y,z)$ and ${\bf X}_\varphi = \pa{\bf X}/\pa\varphi$ and
${\bf X}_\theta = \pa{\bf X}/\pa\theta$ are tangent vectors.  

In general, if
$(\xi_1,\xi_2)$ are coordinates on a surface, the tangent vectors
${\bf X}_i = \pa{\bf X}/\pa\xi_i$, $i = 1,2$ form a basis of the tangent space at each point on the surface. With metric tensor
$g_{ij} = {\bf X}_i\cdot{\bf X}_j$ and inverse $g^{ij} = (g_{ij})^{-1}$,
we have dual tangent vectors ${\bf X}^*_i = \Sigma_j g^{ij}{\bf X}_j $
so that ${\bf X^*_i}\cdot{\bf X_j} = \delta_{ij}$.  For a scalar function $\Phi$
the surface gradient is 
\beq  \nabla\Phi \eq \frac{\pa\Phi}{\pa\xi_1}{\bf X^*_1}
          \p \frac{\pa\Phi}{\pa\xi_2}{\bf X^*_2} \eeq
Note that ${\bf X}_i\cdot\nabla\Phi = \pa\Phi/\pa\xi_i$.
In the specific case $(\xi_1,\xi_2) = (x,y)$, with $z = z(x,y)$,
we have ${\bf X}_1 = (1,0,z_x)^T$ and ${\bf X}_2 = (0,1,z_y)^T$.
For a tangential vector field ${\bf v} = (v_1,v_2,v_3)^T$ in Cartesian form as above,
we have
\beq {\bf v} = v_1{\bf X}_1 + v_2{\bf X}_2 \eeq
 so that
\beq {\bf v}\cdot\nabla\Phi = v_1\Phi_x + v_2\Phi_y \eeq
 leading to the third form
of the p.d.e. (4.8) above, and the other two are similar. 
For later use we note that, with $z = z(x,y)$, the surface gradient is
\beq \nabla\Phi \eq (\Phi_x,\Phi_y,0)^{tan}  \eeq
where $tan$ denotes the tangential part of the vector,
\beq {\bf w}^{tan} \eq {\bf w} \m ({\bf w}\cdot{\bf n}){\bf n}  \eeq
as can be seen by checking the scalar product with ${\bf X}_i$.

The surface divergence of a tangential vector field in general coordinates
on a surface is
\beq \nabla\cdot{\bf v} \eq \frac{\pa{\bf v}}{\pa\xi_1}\cdot{\bf X^*_1}
          \p \frac{\pa{\bf v}}{\pa\xi_2}\cdot{\bf X^*_2} \eeq
This is equivalent to a standard formula,
e.g. in \cite{aris},
$\nabla\cdot{\bf v} = g^{-1/2}\Sigma_i\pa(g^{1/2}v_i)/\pa\xi_i$
with $v_i$ as in (4.17).
For the unit sphere, again with coordinates $(x,y)$ and $z = z(x,y)$,
a calculation gives
\beq \nabla\cdot{\bf v} \eq v_{1,x} + v_{2,y} + (x/z^2)v_1 + (y/z^2)v_2 \eeq
Corresponding remarks apply
to the other coordinate systems, $(x,y)$ or $(z,x)$.

\medskip


\noindent {\bf The shallow water equations.}
The shallow water equations on a sphere are a model for depth-averaged fluid flow with tangential velocity field $\bf v$ and geopotential
$\Phi$.  They are
\begin{gather}
  \frac{d{\bf v}}{dt} \p f{\bf n}\times{\bf v} \p \nabla\Phi \eq 0\\
  \frac{d\Phi}{dt} \p \Phi\nabla\cdot{\bf v} \eq 0
\end{gather}
where  $d/dt$ is the material or substantial derivative,
$\nabla\cdot{\bf v}$ is the surface divergence, $f$ is the
Coriolis parameter, and ${\bf n}$ is the outward unit normal vector.  
The material derivatives are
\beq \frac{d\Phi}{dt} \eq \frac{\pa\Phi}{\pa t} \p {\bf v}\cdot\nabla\Phi \eeq
where $\nabla\Phi$ is the surface gradient, and
\beq \frac{d{\bf v}}{dt} \eq \frac{\pa{\bf v}}{\pa t} \p \nabla_{\bf v}{\bf v} \eeq
where the second term is the covariant derivative on the surface (e.g. \cite{aris}).

The velocity is tangent to the surface, but we will treat it as a Cartesian vector
${\bf v} = (v_1,v_2,v_3)$.  We will describe the equations for part of the surface with coordinates $(x,y)$ and $z = z(x,y)$.
Since  ${\bf v} =  v_1{\bf X}_1 + v_2{\bf X}_2$ and $\bf v$ is Cartesian, the 
covariant derivative of the velocity is simply (\cite{docarmo}, Sec. 4.4)
\beq \nabla_{\bf v}{\bf v} \eq \{v_1{\bf v}_x + v_2{\bf v}_y \}^{tan} \eeq
Substituting equations (4.18),(4.19),(4.22),(4.27) in (4.23),(4.24) we have
\begin{gather}
\frac{\pa{\bf v}}{\pa t} \p \{v_1{\bf v}_x \p v_2{\bf v}_y \}^{tan} \p
     f{\bf n}\times{\bf v} \p (\Phi_x,\Phi_y,0)^{tan} \eq 0  \\
  \frac{\pa\Phi}{\pa t} \p (v_1\Phi_x + v_2\Phi_y) \p
      \Phi\left(v_{1,x} + v_{2,y} + (x/z^2)v_1 + (y/z^2)v_2\right) \eq 0
\end{gather}
We will write the equations with derivative terms in conservation form before
discretizing.  We can combine terms in the $\Phi$-equation to get
\beq \frac{\pa\Phi}{\pa t} \p (\Phi v_1)_x + (\Phi v_2)_y
       \p     (x/z^2)\Phi v_1 + (y/z^2)\Phi v_2 \eq 0  \eeq 
As usual we replace the ${\bf v}$-equation with one for $\Phi{\bf v}$.
Then $\pa(\Phi{\bf v})/\pa t$ includes terms
\beq \Phi\{v_1{\bf v}_x \p v_2{\bf v}_y \}^{tan} \p 
  (\Phi v_1)_x{\bf v} + (\Phi v_2)_y{\bf v}  \eq 
   \{(\Phi v_1{\bf v})_x + (\Phi v_2{\bf v})_y\}^{tan} \eeq
and we obtain the equation
\begin{multline}
\frac{\pa(\Phi{\bf v})}{\pa t} \p
     \{(\Phi v_1{\bf v})_x + (\Phi v_2{\bf v})_y \}^{tan} \p 
       \frac12((\Phi^2)_x,(\Phi^2)_y,0)^{tan} \\ 
\p f{\bf n}\times(\Phi{\bf v}) \p 
  \left((x/z^2)v_1 + (y/z^2)v_2\right)\Phi{\bf v}   \eq 0  
\end{multline}

We use the formulation (4.30), (4.32) to compute the solution
at cut points in $\Gamma_3$.
With a similar treatment for the other two cases, we can solve the system of equations in a manner like that for the advection equation.  Again we use the
MacCormack version of the Lax-Wendroff method. For stability we use a standard
artificial viscosity, e.g. as in \cite{peyret}.  We add terms which for (4.30) approximate
\beq \nu h^2 \left( \left(|\nabla\Phi|\Phi_x\right)_x +  
\left(|\nabla\Phi|\Phi_y\right)_y \right)  \eeq
and similarly for $\Phi{\bf v}$ in (4.32).
For (4.30), in the corrector step corresponding to (4.12), on $\Gamma_3$, we add
\beq \nu kh \sum_i \left(|D^+\Phi^n|D^+_i\Phi^n - |D^-\Phi^n|D^-_i\Phi^n \right) \eeq
with $i = 1,2$, where $D^\pm_i$ is the forward or backward difference in direction
$x$ or $y$ and $|D^\pm\Phi^n| = (|D^\pm_1\Phi^n|^2 + |D^\pm_2\Phi^n|^2)^{1/2}$.
An analogous term is used for $\Phi{\bf v}$.

As a test problem we use the second example from the well-known test set of 
Williamson et al. \cite{testset}, an exact steady solution of (4.23--24) with a replacement
for $f$.  In rectangular
coordinates the formulas are
\begin{gather}
 {\bf v} \eq u_0 (-c_\alpha y, c_\alpha x + s_\alpha z, - s_\alpha y)^T \\
   \Phi \eq \Phi_0 - (a\Omega u_0 + u_0^2/2)(-s_\alpha x + c_\alpha z)^2  \\
 f \eq 2\Omega(-s_\alpha x + c_\alpha z) \,, \qquad 
c_\alpha = \cos{\alpha}\,, s_\alpha = \sin{\alpha}  
\end{gather}
with parameters $u_0$, $a$, $\Omega$ given in \cite{testset} and $\alpha$ an arbitrary angle.
We chose $\alpha = 30^o$ and $h = 2.4/N$, time step $k = 1/2N$ as before.
We set the viscosity coefficient to $\nu = .5$ or $1$.  Results are shown in
Tables 4.2 and 4.3 after $1$, $2$ and $5$ days.  We display the relative errors in maximum norm and in $L^2$,
defined as in (4.13), and the relative errors in the surface integrals of
$|{\bf v}|^2$ and $\Phi$, computed as in Appendix B.  With $\nu = 1$, the $L^2$ errors
in $\Phi{\bf v}$ and $\Phi$ and the maximum error in $\Phi$ are about $O(h^2)$.
The maximum error for $\Phi{\bf v}$
appears between $O(h)$ and $O(h^2)$ for $N \leq 320$ but not to
$N = 640$.  For $\nu = .5$, the $L^2$ errors are mostly somewhat smaller, but the errors are less regular in dependence on $h$.  The discrepancy between $L^2$ and maximum errors is less for $\nu = 1$ than $\nu = .5$.

\begin{table}[ht]
{\footnotesize
\caption{Relative errors for the shallow water equations, $\nu = 1$} 
\begin{center}
\begin{tabular}{| c | c | c | c | c | c | c | c |} \hline
 N & time & max $\Phi{\bf v}$ & max $\Phi$ & $L^2$ in $\Phi{\bf v}$ 
  & $L^2$ in $\Phi$ & $\int |{\bf v}|^2$ & $\int\Phi$ \\ \hline  \hline
 \multirow{3}{*}{80}
 & 1 & 2.47e-2 &  1.22e-2 &  1.62e-2 &  4.65e-3 & -2.70e-2 &  8.68e-4
\\ \cline{2-8}
 & 2 & 3.78e-2  & 2.23e-2  & 3.05e-2  & 9.02e-3 & -5.18e-2  & 1.69e-3
\\ \cline{2-8}
 & 5 & 8.98e-2 &  4.62e-2 &  7.15e-2 &  2.05e-2 & -1.23e-1 & 3.95e-3
\\ \hline \hline
\multirow{3}{*}{160}
 & 1 & 8.01e-3 &  3.26e-3 &  4.22e-3 &  1.20e-3 & -6.86e-3 &  2.29e-4
\\ \cline{2-8}
 & 2 &  1.23e-2 &  6.18e-3  & 8.03e-3 &  2.38e-3 & -1.34e-2 &  4.54e-4
\\ \cline{2-8}
 & 5 &  2.69e-2 &  1.38e-2 &  1.96e-2 &  5.68e-3 & -3.32e-2 &  1.11e-3
\\ \hline \hline
\multirow{3}{*}{320}
 & 1  & 2.26e-3 &  8.34e-4 &  1.07e-3  & 3.02e-4 & -1.72e-3 &  5.88e-5
\\ \cline{2-8}
 & 2  & 3.43e-3 & 1.60e-3  & 2.04e-3 &  6.04e-4 & -3.36e-3 &  1.17e-4
\\ \cline{2-8}
 & 5 & 7.10e-3 &  3.68e-3 &  5.05e-3 &  1.46e-3 & -8.47e-3 &  2.92e-4
\\ \hline \hline
\multirow{3}{*}{640}
 & 1 & 2.17e-3  & 2.09e-4 &  2.70e-4  & 7.57e-5 & -4.32e-4  & 1.45e-5
\\ \cline{2-8}
 & 2  & 4.37e-3  & 4.01e-4  & 5.29e-4 &  1.53e-4 & -8.63e-4 &  2.46e-5
\\ \cline{2-8}
 & 5  &  4.83e-3 &  9.31e-4  & 1.32e-3 &  3.76e-4 & -2.21e-3  & 5.55e-5
\\ \hline 
\end{tabular} \end{center}  }
\end{table}

\begin{table}[ht]
{\footnotesize
\caption{Relative errors for the shallow water equations, $\nu = .5$} 
\begin{center}  
\begin{tabular}{| c | c | c | c | c | c | c | c |} \hline
N & time & max $\Phi{\bf v}$ & max $\Phi$ & $L^2$ in $\Phi{\bf v}$ 
  & $L^2$ in $\Phi$ & $\int |{\bf v}|^2$ & $\int\Phi$ \\ \hline  \hline
  \multirow{3}{*}{80}
  & 1 & 1.49e-2 &  6.49e-3  & 9.01e-3 &  2.44e-3 & -1.47e-2 &  3.63e-4
\\ \cline{2-8}
  & 2 & 2.27e-2 &  1.22e-2 &  1.68e-2  & 4.85e-3 & -2.80e-2 &  7.13e-4
\\ \cline{2-8}
  & 5 & 5.30e-2 & 2.63e-2 &  4.05e-2 &  1.13e-2 & -6.85e-2 &  1.70e-3
\\ \hline \hline
\multirow{3}{*}{160}
  & 1 & 4.61e-3  & 1.67e-3 &  2.30e-3  & 6.19e-4 & -3.71e-3 &  9.61e-5
\\ \cline{2-8}
  & 2 & 8.46e-3 &  3.19e-3  & 4.34e-3 &  1.25e-3 & -7.12e-3  & 1.86e-4
\\ \cline{2-8}
 & 5 &  1.86e-2 &  7.22e-3 &  1.08e-2 & 3.03e-3 & -1.80e-2  & 4.03e-4
\\ \hline \hline
\multirow{3}{*}{320}
 & 1 & 4.48e-3 &  4.19e-4  & 5.83e-4 &  1.55e-4 & -9.30e-4 &  2.43e-5
\\ \cline{2-8}
 & 2 &  1.31e-2  & 9.78e-4  & 1.13e-3 &  3.15e-4 & -1.78e-3 & 4.81e-5
\\ \cline{2-8}
 & 5 &  1.50e-2 &  1.86e-3 &  2.75e-3 &  7.59e-4 & -4.49e-3 &  1.29e-4
\\ \hline \hline
\multirow{3}{*}{640}
 & 1 & 1.37e-2  & 7.71e-4 &  3.00e-4  & 4.26e-5 & -2.68e-4 & -4.23e-7
\\ \cline{2-8}
 & 2 & 1.39e-2  & 8.18e-4 &  4.48e-4 & 9.05e-5 & -5.24e-4 & -7.60e-6
\\ \cline{2-8}
 & 5 & 1.40e-2  & 9.33e-4 &  9.31e-4  & 2.72e-4 & -1.38e-3 & -2.87e-5
\\ \hline 
\end{tabular} \end{center} }
\end{table}

\ssection {\bf Discussion} 

We have seen that, for a variety of partial differential equations, the present method of surface discretization permits the use of conventional finite difference methods for planar regions.  The formulation of the diffusion equation in Section 3 could be extended to reaction-diffusion equations as e.g. in
\cite{zhao13,closepnas,rbffogel}, combining diffusion equations with nonlinear ordinary differential equations.  We have used quadratic interpolation, but higher order interpolation could be used.  This might be needed for equations with higher order diffusion.  It should be possible to extend
the method to moving surfaces represented by level set functions.

In Section 4 we applied this method to the shallow water equations on a sphere.  The form of the equations in the coordinate systems is fairly straightforward.  We found the Lax-Wendroff method was adequate for the simple test problem considered.  For more realistic problems Riemann solvers might be used.  The semi-Lagrangian method has been successfully used in meteorology; e.g. see \cite{durran}.  In this approach,
to approximate the material derivative, values of the unknowns are obtained at the new time by following a particle path backward in time to find the previous value at the departure point along the path.  The old value must be interpolated from values at grid points.  We expect that such a strategy can be used with the approach presented here.  We would need to find the primary point closest to the departure point and interpolate in a two-dimensional neighborhood.
\appendix

\ssection {\bf
Positivity of the resolvent of the LB operator on a curve}
For a closed curve in $\R^2$ the Laplace-Beltrami operator, acting on a function $u$ on the curve, is
\beq \Delta u \eq c \frac{d}{d\xi}\left(c\frac{du}{d\xi}\right)\,,
         \quad c \eq \left|\frac{d{\bf X}}{d\xi}\right|^{-1} \eq 
           \left(\frac{ds}{d\xi}\right)^{-1}  \eeq
where $\xi$ is a coordinate and ${\bf X} = (x,y)$.  Thus $\Delta$ is the second arclength derivative.  With cut points selected as before in sets $\Gamma_1$, $\Gamma_2$, we can discretize in divergence form as in (3.6),(3.7)
and obtain an expression for the discrete Laplacian at a primary point, e.g. in $\Gamma_2$, with $\xi = x$,
\beq \Delta_h u_i \eq \left( c_{i+1}u_{i+1} + c_{i-1}u_{i-1} - 2c_iu_i\right)/h^2 \,, 
        \quad 2c_i = c_{i+1} + c_{i-1} \eeq
where $c_i$ approximates $1/|{\bf X}_x|^2 = 1/(1 + y_x^2)$ etc.  Assuming the curve is $C^2$, and using Lemma 2.1, we have
\beq c_i \geq 1/2 -  O(h)\,, \quad c_{i\pm 1} - c_i = O(h)  \eeq
We can form the reduced LB operator $\Delta_h^{red}$ as before.  We prove the positivity property described earlier for the resolvent.

\begin{lemma}
For $k/h^2 \geq 1/2 + O(h)$, the matrix
$I - k\Delta_h^{red}$ is invertible,  and the inverse has nonnegative entries.
\end{lemma}

\begin{proof}
  Let $n_p$ and $n_s$ be the number of primary and secondary points, resp.  We will use an 
$(n_p + n_s)\times (n_p + n_s)$ matrix $A$ which incorporates $\Delta_h$ evaluated
at the primary points and the
interpolation (2.1) for secondary points.  The upper part of $A$ is the $n_p \times (n_p + n_s)$ matrix
formed by $I - k\Delta_h$ at the primary points.  Thus the $i$th row, corresponding to the case above
has nonzero entries $a_{i,i} = 1 + 2\sigma c_i$, $a_{i,i\pm 1} = - \sigma c_{i\pm 1}$, with $\sigma = k/h^2$.
The lower part of $A$ consists of the $n_s\times (n_p + n_s)$ matrix taking the column vector
$(u^p \; u^s)^T$ to $u^s - \Pi^{sp}u^p - \Pi^{ss}u^s$ with $\Pi^{sp}$, $\Pi^{ss}$ as in (2.2).  According to (2.1), the $\ell$th row, with $n_p+1 \leq \ell \leq (n_p + n_s)$, has nonzero entries $a_{\ell,\ell} = 1$,
$a_{\ell,i} = -(1-\theta_\ell^2)$, $a_{\ell,i\pm 1} = (\mp\theta_\ell - \theta_\ell^2)/2$ with
$|\theta_\ell| \leq 1/2$, where $i$ is the index for the primary point associated with the secondary point
of index $\ell$.

The matrix $A$ has positive diagonal entries and nonpositive entries off-diagonal, except that one entry
in the $\ell$th row, say $a_{\ell,i+1}$, could be positive, with $0 \leq a_{\ell,i+1} \leq 1/8$.  In this sense, $A$ is close to being an $M$-matrix.  We will modify it by row operations to obtain an $M$-matrix.
Consider rows $\ell$ and $i$ as above.  We add $r_\ell$ times row $i$ to row $\ell$, with 
$r_\ell = 1/(8\sigma c_{i+1})$ so that the new $(\ell,i+1)$ entry is $\leq 0$.  We need to ensure that
the new $(\ell,i)$ entry is also $\leq 0$.  Since the original entry is $\leq - 3/4$, we must require
$r_\ell(1 + 2\sigma c_i) \leq 3/4$ or $(1 + 2\sigma c_i)/(\sigma c_{i+1}) \leq 6$.
In view of (A.3), this is true provided we assume
$   \sigma \geq 1/2 + O(h)  $.  Thus we have eliminated the off-diagonal entry $\geq 0$ in row $\ell$.
This row operation can be accomplished by premultiplying $A$ by the matrix $P^{(\ell)}$ 
where $P^{(\ell)}$ has $1$'s on the diagonal, $r_\ell$ in the $(\ell,i)$ entry and zero otherwise.
Repeating the same procedure for each row with $\ell > n_p$, we can obtain a matrix $P$ with nonnegative
entries so that $PA$ has diagonal entries $\geq 0$ and off-diagonal entries $\leq 0$.  The row sums are
$1$ in the upper part and $r_\ell$ in the $\ell$-th row in the lower part.  Thus $PA$ is strictly 
diagonally dominant and is an $M$-matrix.  Consequently it is invertible, and the inverse has entries $\geq 0$ (e.g. see \cite{hackbook}).

We now have $(PA)^{-1} \geq 0$ entrywise.  To return to $A$, note that
$(PA)^{-1}PA = I$ so that $(PA)^{-1}P = A^{-1}$.  This is a product of matrices
$\geq 0$, so $A^{-1} \geq 0$.  (Cf. \cite{bramble}, Thm. 2.3.)
We can now relate the inverse of $I - k\Delta_h^{red}$ to that of $A$.
Given a vector $y^p$ of length $n_p$, let $(u^p\; u^s)^T = A^{-1}(y^p \; 0)^T$. 
Since the secondary part of $A(u^p\; u^s)^T$ is zero, $u^s$ is equilibrated.
It follows that $(I - k\Delta_h^{red})u^p = (I - k\Delta_h)(u^p\; u^s)^T = y^p$,
and thus $I - k\Delta_h^{red}$ is invertible.  If $y^p \geq 0$, it follows that
$u^p \geq 0$ since $A^{-1} \geq 0$.  We have shown that $y^p \geq 0$ implies
$(I - k\Delta_h^{red})^{-1}y^p \geq 0$, and thus
$ (I - k\Delta_h^{red})^{-1} \geq 0$. 
\end{proof}

\ssection {\bf A quadrature rule for surface integrals}
We describe a quadrature rule for surface integrals, using the sets
$\Gamma_v$ of cut points, introduced in \cite{Wilson} and explained in \cite{byw}.
Quadrature weights are defined from a partition of unity on the unit sphere,
applied to the normal vector to the surface.
Suppose $\Gamma$ is a closed surface and the sets $\Gamma_\nu$, $\nu = 1,2,3$,
are defined as in Sec. 2 with the restriction $|n_\nu| \geq \eta$.  We choose
$\eta$ and an angle $\theta$ so that $\eta < \cos{\theta} < 1/\sqrt{3}$,
e.g., $\eta = .45$ and $\theta = 62.5^o$.  We define a partition of unity
on the unit sphere $S$ using the bump function $b(r) = \exp(r^2/(r^2-1))$ for
$|r| <  1$ and $b(r) = 0$ otherwise.  For ${\bf n} \in S$ define for $\nu = 1,2,3$
\beq \sigma_\nu({\bf n}) \eq b((\cos^{-1}{|n_\nu|})/\theta)\,, \qquad
      \psi_\nu({\bf n}) \eq \sigma_\nu({\bf n})/\left(\sum_{j=1}^3 \sigma_j({\bf n}) \right) \eeq
Then $\Sigma \psi_\nu \equiv 1$ on $S$.  The quadrature rule is
\beq \int_\Gamma f({\bf x})\,dS({\bf x}) \,\approx\, 
     \sum_{\nu=1}^3 \sum_{{\bf x}\in\Gamma_\nu} 
      f({\bf x})\,\psi_\nu({\bf n}({\bf x}))|n_\nu({\bf x})|^{-1}h^2  \eeq
It is high order accurate, i.e., the accuracy is limited only by the smoothness of
the surface $\Gamma$ and the integrand $f$.

\section*{Acknowledgments}  We are grateful to Wenjun Ying for discussions about the surface discretization and to Thomas Witelski for several suggestions.

\bibliographystyle{amsplain}

\end{document}